\documentclass[12pt,a4paper]{amsart}

\newtheorem{theo+}              {Theorem}           [section]
\newtheorem{prop+}  [theo+]     {Proposition}
\newtheorem{coro+}  [theo+]     {Corollary}
\newtheorem{lemm+}  [theo+]     {Lemma}
\newtheorem{exam+}  [theo+]     {Example}
\newtheorem{rema+}  [theo+]     {Remark}
\newtheorem{defi+}  [theo+]     {Definition}
\def \r{\mbox{${\mathbb R}$}}

\newenvironment{theorem}{\begin{theo+}}{\end{theo+}}
\newenvironment{proposition}{\begin{prop+}}{\end{prop+}}
\newenvironment{corollary}{\begin{coro+}}{\end{coro+}}
\newenvironment{lemma}{\begin{lemm+}}{\end{lemm+}}

\usepackage{amsthm}
\theoremstyle{plain} \theoremstyle{remark}
\newtheorem{remark}{Remark}
\newtheorem{example}{Example}

\def\E{/\kern-1.0em \equiv }

\title{$f$-biharmonic Riemannian submersions from 3-manifolds}

\author{Ze-Ping Wang$^{*}$ and Li-Hua Qin }

\address{School of Mathematical Sciences,\newline\indent Guizhou
Normal University,\newline\indent Guiyang 550025,\newline\indent
People's Republic of China
\newline\indent E-mail:zpwzpw2012@126.com \;(Wang)\\\newline\indent
E-mail:1727729537@qq.com \;(Qin)}
\thanks{*Supported by the Scientific and Technological Project in Guizhou Province ( Grant no. Qiankehe Platform Talents [2018]5769-04) and  by the NSFC (No. 11861022). }
\date{2/1/2024}
\begin{document}

\title[$f$-Biharmonic submanifolds  and  $f$-biharmonic Riemannian submersions ] {$f$-Biharmonic submanifolds in space forms and  $f$-biharmonic Riemannian submersions from 3-manifolds }

\subjclass{58E20, 53C12} \keywords{Biharmonic maps, $f$-biharmonic maps, Riemannian submersions,  $f$-biharmonic curves, $f$-biharmonic submanifolds.}

\maketitle

\section*{Abstract}
\begin{quote}
{\footnotesize  $f$-Biharmonic maps are generalizations of harmonic maps and  biharmonic maps.
 In this paper, we  obtain some  descriptions of   $f$-biharmonic curves  in a space form.  We also obtain a complete classification of  proper  $f$-biharmonic isometric immersions of a developable surface in $\r^3$ by proving that  a proper $f$-biharmonic developable surface  exists only  in the case where the surface is a cylinder. Based on this, we  show that a proper biharmonic conformal immersion of a developable surface into $\r^3$ exists only  in the case when the surface is a cylinder. Riemannian submersions can be viewed as the dual notion of  isometric immersions (i.e., submanifolds). We also study $f$-biharmonicity of  Riemannian submersions from 3-space forms by using the  integrability data. Examples are given of proper $f$-biharmonic Riemannian submersions and $f$-biharmonic surfaces and curves.}
\end{quote}

\section{Introduction and Preliminaries}
All manifolds, maps, tensor fields studied in this paper are
assumed
to be smooth unless there is an otherwise statement.\\

A {\em biharmonic map} is a map $\varphi:(M, g)\longrightarrow (N,
h)$ between Riemannian manifolds that is  a critical point of the
bienergy
\begin{equation}\nonumber
E^{2}\left(\varphi,\Omega \right)= \frac{1}{2} {\int}_{\Omega}
\left|\tau(\varphi) \right|^{2}v_g
\end{equation}
for every compact subset $\Omega$ of $M$, where
$\tau(\varphi)={\rm Trace}_{g}\nabla {\rm d} \varphi$ is the
tension field of $\varphi$ vanishing of which means the map is
harmonic. By computing the first variation of the functional (see
\cite{Ji}) we see that $\varphi$ is  biharmonic  if and only if
its bitension field vanishes identically, i.e.,
\begin{equation}\label{BT1}
\tau_{2}(\varphi):={\rm
Trace}_{g}(\nabla^{\varphi}\nabla^{\varphi}-\nabla^{\varphi}_{\nabla^{M}})\tau(\varphi)
- {\rm Trace}_{g} R^{N}({\rm d}\varphi, \tau(\varphi)){\rm
d}\varphi =0,
\end{equation}
where $R^{N}$ is the curvature operator of $(N, h)$ defined by
$$R^{N}(X,Y)Z=
[\nabla^{N}_{X},\nabla^{N}_{Y}]Z-\nabla^{N}_{[X,Y]}Z.$$
{\bf $f$-biharmonic maps and their equations}: $f$-biharmonic maps are critical
points of the $f$-bienergy functional for maps $\varphi : (M, g) \to(N, h)$ between Rie-
mannian manifolds:
\begin{equation}\notag
E_{2,f} (\varphi) =\int_{\Omega}f|\tau(\varphi)|^2v_g,
\end{equation}
where $\Omega$
 is a compact domain of $M$  and $f:M\to(0,+\infty)$. The Euler-Lagrange equation gives the
$f$-biharmonic map equation ( see e.g., \cite{Lu,Ou10})
\begin{equation}\label{fbeq}
\tau_{2,f}(\varphi) =-J^{\varphi}(f\tau(\varphi))=f\tau_{2}(\varphi)+(\Delta f)\tau(\varphi)+2\nabla^{\varphi}_{{\rm grad}f}\tau(\varphi)=0,
\end{equation}
where $\tau(\varphi)$ and $\tau_2(\varphi)$ are the tension and bitension fields of $\varphi$ respectively,  $J^{\varphi}$ is the  Jacobi operator of the map defined by $J^{\varphi}(X)=-\{{\rm
Trace}_{g}(\nabla^{\varphi}\nabla^{\varphi}-\nabla^{\varphi}_{\nabla^{M}})X
- {\rm Trace}_{g} R^{N}({\rm d}\varphi, X){\rm
d}\varphi \}$.  Clearly, both harmonic maps and biharmonic maps are $f$-biharmonic.  We  also have the following inclusion relations\\

$\{Harmonic\; \;maps\} \;\subset\;\{Biharmonic\;\; maps\}$$\;\subset\;\{f$-$biharmonic\; \;maps\}.$\\

 A submanifold in a Riemannian manifold is called an $f$-biharmonic submanifold if the isometric immersion defining the submanifold is an $f$-biharmonic
map. Similarly, a Riemannian submersion is called an $f$-biharmonic
Riemannian submersion if the Riemannian submersion is  an $f$-biharmonic
map. We call an $f$-biharmonic map (respectively,\; submaniold,\; Riemannian submersion ) which is not biharmonic a proper $f$-biharmonic map (respectively,\; submanifold,\; Riemannian submersion ).\\

$f$-Biharmonic map was first introduced in \cite{Lu}. Later, $f$-biharmonic submanifolds  were studied in \cite{Ou10} where
the author  obtained  the $f$-biharmonic curve equation in a space form and also gave  a complete classification of $f$-biharmonic curves in
$\r^3$. The paper \cite{WO} shows that neither of  a circular cone  and a part of  the standard sphere $S^2$ in $\r^3$ is $f$-biharmonic for any $f$, and
a constant mean curvature surface in $\r^3$ is $f$-biharmonic if and only if it is a part of a plane or a circular cylinder. For some recent progress on $f$-biharmonic submanifolds, we refer the readers to \cite{Ou10, Ou9, Ou2} and the references therein. We will focus  to classify (or construct)  proper $f$-biharmonic submanifolds in a space form.\\

  As the dual notion of biharmonic isometric immersions (i.e., biharmonic  submanifolds),  biharmonic Riemannian submersions
 were first studied by using the integrability data of an adapted frame of the Riemannian submersion in \cite{WO} where
the  authors showed that a biharmonic Riemannian submersion from a 3-space form  into a surface  has to be harmonic and also
constructed a family of proper biharmonic Riemannian submersions from $\r^3$ provided with a
warped product metric. Following \cite{WO}, we study $f$-biharmonic  Riemannian submersions from  3-manifolds by using the integrability data.
 However, it is extremely difficult to find a proper $f$-biharmonic  Riemannian submersion
 since the function $f$ is a solution of $PDE$. We will focus  to   construct  examples of proper $f$-biharmonic  Riemannian submersions. \\

In this paper, we first derive $f$-biharmonic curve equation in a general Riemannian manifold. We then obtain some  characterizations of  $f$-biharmonic curves in a space form  by giving  the explicit functions $f$ and the explicit curvatures of the curves.  For $f$-biharmonic surfaces in $\r^3$,
we obtain  a complete classification of proper $f$-biharmonic isometric immersions of a developable surface in $\r^3$ by proving that  a proper $f$-biharmonic developable surface  exists only in the case when the surface is a cylinder. An interesting result is that a cylinder whose directrix  takes  a proper $\bar{f}$-biharmonic curve in a 2-space form $N^2(K)$ of constant Gauss curvature $K$ has to be a proper $f$-biharmonic cylinder for $f= \frac{\psi(v,K)}{c_1}\bar{f}$, where $\psi(v,K)$ given by (\ref{cy8}) and  a constant $c_1>0$.
 We then  construct  many examples of proper $f$-biharmonic cylinders in $\r^3$.  Based on this, we show that a proper biharmonic conformal immersion of a developable surface into $\r^3$ exists only in the case where the surface is a cylinder.
Finally, we  derive $f$-biharmonic  equation for  Riemannian submersions  from 3-manifolds by using the integrability data. We then use it to study $f$-biharmonic Riemannian submersions from a 3-space form. Examples are given of  proper $f$-biharmonic Riemannian submersions.

\section{ $f$-Biharmonic submanifolds in  space forms }
 In this section, we  characterize  $f$-biharmonic curves  in a space form by using the explicit functions $f$  and the explicit curvatures of the curves.
 Many examples of proper $f$-biharmonic curves  in $\r^n$ are obtained.
We also give  complete classifications of $f$-biharmonic  isometric immersions and biharmonic  conformal immersions of a developable surface into $\r^3$.
An interesting result is that a cylinder whose directrix  takes  a proper $\bar{f}$-biharmonic curve in a 2-space form $N^2(K)$ of constant Gauss curvature $K$ has to be a proper $f$-biharmonic cylinder for $f= \frac{\psi(v,K)}{c_1}\bar{f}$, where $\psi(v,K)$ given by (\ref{cy8}) and  a constant $c_1>0$.  Based on this, one constructs infinitely many examples of
proper $f$-biharmonic cylinders  and biharmonic  conformal immersions of  cylinders into $\r^3$
\subsection{$f$-Biharmonic curves in  space forms }

 A Frenet frame $\{F_{i}\}_{i=1, 2, \ldots, n}$ associated to an arclength  parametrized curve $\gamma: (a,b)\to(N^{n}, h)$ (see e.g., \cite{La}) is an orthonormal frame which can be described by
\begin{equation}\notag
\begin{cases}
F_{1}=d\gamma(\frac{\partial}{\partial s})=\gamma',\\
\nabla_{F_1}F_{1}=k_{1} F_{2},\\
\nabla_{F_1}F_{i}=-k_{i-1}
F_{i-1}+k_{i} F_{i+1},\forall\; i=2,\;3,\;\ldots,\;n-1,\\
\nabla_{F_1}F_{n}=-k_{n-1}
F_{n-1},\\
\end{cases}
\end{equation}
where the functions $\{k_{1}, k_{2}, \ldots, k_{n-1}\}$ are called
the curvatures of $\gamma$.  It is well known  (see \cite{Ou10}) that the curve is an $f$-biharmonic curve  with a function $f:(a,b)\to(0,+\infty)$ if and only if
\begin{equation}\label{fn1}
f(\nabla_{\gamma'}^{N}\nabla_{\gamma'}^{N}\nabla_{\gamma'}^{N}\gamma'-R^{N}(\gamma',\nabla_{\gamma'}^{N}\gamma')\gamma')
+2f'\nabla_{\gamma'}^{N}\nabla_{\gamma'}^{N}\gamma'+f''\nabla_{\gamma'}^{N}\gamma'=0.
\end{equation}
With respect to the Frenet frame, the equation (\ref{fn1})  can be written as
\begin{lemma}\label{PE1}
Let $\gamma: (a,b)\to(N^{n}, h) (n\geq 2)$
be a curve parametrized by arc length into a Riemannian manifold. Then $\gamma$
is $f$-biharmonic if and only if :
\begin{equation}\label{pe2}
\left\{\begin{array}{rl}
-3\kappa_1\kappa'_1-2\kappa_1^2f'/f=0,\\
\kappa_1''-\kappa_1\kappa_2^2-\kappa_1^3+\kappa_1 R^{N}(F_{1},F_{2},F_{1},F_{2})+\kappa'_1f''/f+2\kappa'_1f'/f=0,\\
2\kappa_1\kappa_2+\kappa_1\kappa'_2+\kappa_1 R^{N}(F_{1},F_{2},F_{1},F_{3})+2\kappa_1\kappa_2f'/f=0,\\
\kappa_1\kappa_2\kappa_3+\kappa_1 R^{N}(F_{1},F_{2},F_{1},F_{4})=0,\\
k_{1}R(F_{1},F_{2},F_{1},F_{j})=0,\;\;\;\;\;\;j=5,\ldots,n.\;\;
\end{array}\right.
\end{equation}
\end{lemma}
\begin{proof}
A straightforward computation gives
\begin{equation}\label{pe3}
\begin{array}{lll}
\tau(\gamma)=\nabla_{\gamma'}^{N}{\gamma'}=\kappa_1F_2,\\
\nabla_{\gamma'}^{N}\nabla_{\gamma'}^{N}{\gamma'}=-\kappa_1^2F_1+\kappa'_1F_2+\kappa_1\kappa_2F_3,\\
\tau_2(\gamma)=\nabla_{\gamma'}^{N}\nabla_{\gamma'}^{N}\nabla_{\gamma'}^{N}{\gamma'}-R^{N}(\gamma',\nabla_{\gamma'}^{N}{\gamma'})\gamma'\\
=-3\kappa_1\kappa'_1F_1+(\kappa''_1-\kappa_1\kappa_2^2-\kappa_1^3+\kappa_1 R^{N}(F_{1},F_{2},F_{1},F_{2}))F_2\\+(2\kappa'_1\kappa_2
+\kappa_1\kappa'_2+\kappa_1 R^{N}(F_{1},F_{2},F_{1},F_{3}))F_3\\+(\kappa_1\kappa_2\kappa_3+\kappa_1 R^{N}(F_{1},F_{2},F_{1},F_{4}))F_4+
\sum\limits_{j=5}^{n}\kappa_1 R^{N}(F_{1},F_{2},F_{1},F_{j})F_j.\\
\end{array}
\end{equation}
Substituting  (\ref{pe3}) into $f$-biharmonic curve equation (\ref{fn1}) and comparing the coefficients of both sides we obtain (\ref{pe2}).
From which  the lemma follows.
\end{proof}
Applying Lemma \ref{PE1}, we have
\begin{proposition}(see \cite{Ou10})\label{fc0}
An arclength  parametrized curve $\gamma:(a,b)\to N^n(C)$  in an $n$-dimensional space form is an $f$-biharmonic curve  if and only if
one of the following cases happens:\\
(i) $\kappa_2=0, f=c_1\kappa_1^{-3/2}$ and the curvature $\kappa_1$ solves the following $ODE$
\begin{equation}\label{fc1}
3\kappa'^2_1-2\kappa_1\kappa''_1=4\kappa_1^2(\kappa_1^2-C).
\end{equation}
(ii) $\kappa_2\neq0, \;\kappa_3=0,\; \kappa_2/\kappa_1=c_3,\;f=c_1\kappa_1^{-3/2}$ and the curvature $\kappa_1$ solves the following $ODE$
\begin{equation}\label{fc2}
3\kappa'^2_1-2\kappa_1\kappa''_1=4\kappa_1^2[(1+c_3^2)\kappa_1^2-C].
\end{equation}
\end{proposition}
 It is   critical   to solve  the $ODEs$ (\ref{fc1}) and (\ref{fc2}),  to describe $f$-biharmonic curves in space forms.
So we need the following  proposition.
 \begin{proposition}\label{ODE}
 For constants $A$ and $C$,  solving the following $ODE$
\begin{equation}\label{ode}
3y'^2-2yy''=4y^2(Ay^2-C),
\end{equation}
we obtain all nonconstant solutions as
\begin{equation}\label{RH0}\notag
y=\begin{cases}
\frac{1}{C_1e^{2\sqrt{-C}s}+C_2e^{-2\sqrt{-C}s}\pm \sqrt{4C_1C_2+\frac{A}{C}}},\;\;\;\;\;\;\;\;{\rm for}\; C<0,\\
\frac{4C_1}{16A+C_1^2(s+C_2)^2},\;\;\;\;\;\;\;\;\;\;\;\;\;\;\;\;\;\;\;\;\;\;\;\;\;\;\;\;\;\;{\rm for}\; C=0,\\
\frac{1}{C_1\cos(2\sqrt{C}s)+C_2\sin(2\sqrt{C}s)\pm \sqrt{C_1^2+C_2^2+\frac{A}{C}}},\;{\rm for}\; C>0,
\end{cases}
\end{equation}
where $C_1$, $C_2$ are constants and $C_1^2+C_2^2\neq0$.
\end{proposition}

\begin{proof}

First of all, one sees that $y=0$ is a solution of  (\ref{ode}). For $AC>0$,
it is easy to check that  (\ref{ode}) has constant solutions $y=\pm\sqrt{\frac{C}{A}}$.

From now on, we only need to consider that  $y(s)$ is a nonconstant solution of (\ref{ode}). Putting $y=u^{-1}$ and substituting this into  (\ref{ode}), we have
  \begin{equation}\label{bb1}
u'^2-2uu''=4(Cu^2-A).
\end{equation}
Setting $p=u'=\frac{du}{ds}$, then we have $p\frac{dp}{du}=u''(s)$. Denoting by $\frac{dp}{du}=p'$, then (\ref{bb1}) turns into
 \begin{equation}
(p^2)'-\frac{1}{u}p^2+4Cu-4A/u=0,
\end{equation}
 which is solved by $p^2=-4Cu^2+Bu-4A.$ This, together with $u'=p$,  implies that
\begin{equation}\label{bb2}
u'^2=-4Cu^2+Bu-4A,
\end{equation}
where $B$ is a constant. \\
We take the derivative of  both sides of (\ref{bb2}) with respect
to $s$ and simplify the resulting equation  to obtain
\begin{equation}\label{bb4}
u''=-4Cu+\frac{1}{2}B.
\end{equation}
For the case $C>0$,  we solve (\ref{bb4}) to obtain the general solution as
\begin{equation}\label{bb7}
\begin{array}{lll}
u=C_1\cos(2\sqrt{C}s)+C_2\sin(2\sqrt{C}s)+\frac{B}{8C},
\end{array}
\end{equation}
where $C_1, C_2$ are constants. Substituting (\ref{bb7}) into (\ref{bb2}) and simplifying the resulting equation we have
\begin{equation}\label{bb8}
\begin{array}{lll}
B^2=64AC+64C^2(C_1^2+C_2^2),
\end{array}
\end{equation}
which implies that
\begin{equation}\notag
\begin{array}{lll}
u=C_1\cos(2\sqrt{C}s)+C_2\sin(2\sqrt{C}s)\pm \sqrt{\frac{A}{C}+(C_1^2+C_2^2)},\\
({\rm and \;hence})\;y=\frac{1}{C_1\cos(2\sqrt{C}s)+C_2\sin(2\sqrt{C}s)\pm \sqrt{\frac{A}{C}+C_1^2+C_2^2}},
\end{array}
\end{equation}
where constant $8C\sqrt{\frac{A}{C}+C_1^2+C_2^2}=\pm B$.

In a similar way,  we solve (\ref{ode}) for  $C=0$ and $C<0$ respectively, to obtain $y=\frac{4C_1}{16A+C_1^2(s+C_2)^2}$  and
$y=\frac{1}{C_1e^{2\sqrt{-C}s}+C_2e^{-2\sqrt{-C}s}\pm\sqrt{4C_1C_2+\frac{A}{C}}} $ respectively,
where $C_1$ and  $C_2$  are constants.

Summarizing all results above we obtain the proposition.
\end{proof}
\begin{remark}\label{2}
Hereafter, $c_1>0$,\;$C$,\;$c_3$,\; $C_3>0$,\;$C_4$,\;$C_5$,\; $C_1>0$,\;$C_2>0$ and\;$4C_1C_2+\frac{1+c_3^2}{C}>0$ are assumed to be constant unless it is otherwise stated. It is convenient to introduce the following new function
\begin{equation}\label{RH0}
\chi(s,c_3,C)=\begin{cases}
\frac{1}{C_1e^{2\sqrt{-C}s}+C_2e^{-2\sqrt{-C}s}+\sqrt{4C_1C_2+\frac{1+c_3^2}{C}}},\;\;\;\;\;\;\;\;{\rm for}\; C<0,\\
\frac{4C_3}{16(1+c_3^2)+C_3^2(s+C_4)^2},\;\;\;\;\;\;\;\;\;\;\;\;\;\;\;\;\;\;\;\;\;\;\;\;\;\;\;\;\;\;{\rm for}\; C=0,\\
\frac{1}{C_4\cos(2\sqrt{C}s)+C_5\sin(2\sqrt{C}s)+ \sqrt{\frac{1+c_3^2}{C}+(C_4^2+C_5^2)}},\;{\rm for}\; C>0.
\end{cases}
\end{equation}
\end{remark}
As an application of Proposition \ref{ODE}, we now give a characterization of  a proper $f$-biharmonic curve in $ N^n(C)$.
\begin{theorem}\label{RSH}
An arclength parametrized curve $\gamma:(a,b)\to N^n(C)$   in an $n$-dimensional  space form is a proper $f$-biharmonic curve if and only if
one of the following cases happens:\\
$(i)$: $\gamma$ is a curve with $\kappa_2=0$,\;$\kappa_1=\chi(s,0,C)$ and $f=c_1\kappa_1^{-3/2}$,\; or\\
$(ii)$:  $\gamma$ is a curve with $\kappa_3=0,\; \kappa_2/\kappa_1=c_3\neq0$, $\kappa_1=\chi(s,c_3,C)$ and \;$f=c_1\kappa_1^{-3/2}$.
\end{theorem}
\begin{proof}
Applying Proposition \ref{fc0} and Proposition \ref{ODE} with $A=1$ or $A=1+c_3^2$ to (\ref{ode}), we immediately obtain the theorem.
\end{proof}
\begin{remark}\label{02}
From Theorem \ref{RSH},  we can give the explicit function $f$ for any proper $f$-biharmonic curve in a  space form, in a sense, our result recovers Theorem 4.2 and 4.4 in \cite{Ou10}.
\end{remark}

By Theorem \ref{RSH}, we see that a proper $f$-biharmonic curve in $\r^3$ is either a planar curve or a general helix.

\begin{corollary}( \cite{Ou10})\label{Th3}
An arclength  parametrized curve $\gamma:(a,b)\to\r^3$  is a proper $f$-biharmonic curve  if and only if
one of the following cases happens:\\
$(i)$ $\gamma$ is  a planar curve with $\kappa_2=0$, $\kappa_1(s)= \frac{4C_1}{16+C_1^2(s+C_2)^2}$, and $f=c_1\kappa_1^{-3/2}$,\; or\\
$(ii)$ $\gamma$ is a general helix with $\kappa_1(s)= \frac{4C_1}{16(1+c_3^2)+C_1^2(s+2C_2)^2}$, $\kappa_2/\kappa_1(s)= c_3\neq0$, and $f=c_1\kappa_1^{-3/2}$.
\end{corollary}
\begin{remark}\label{002}
A proper $f$-biharmonic map $\varphi: (M, g) \to(N, h)$ followed by a totally geodesic embedding $\psi: (N, h)\to(Q,k)$ is  proper $f$-biharmonic map
$\psi\circ\varphi: (M, g) \to(Q, k)$.

 In fact, using the fact in \cite{Ou8}\;(Page 371), we can easily check that $\tau(\psi\circ\varphi)=d\psi(\tau(\varphi))$,
 $\tau^2(\psi\circ\varphi)=d\psi\left(\tau^2(\varphi)\right)$ and $\nabla^{\psi\circ\varphi}_{{\rm grad} f}\tau(\psi\circ\varphi)=d\psi\left(\nabla^{\varphi}_{{\rm grad} f}\tau(\varphi)\right)$. It follows that $\tau_{2,f}(\psi\circ\varphi)=d\psi\left(\tau_{2,f}(\varphi)\right)$ and hence $\tau_{2,f}(\psi\circ\varphi)=0$ iff $\tau_{2,f}(\varphi)=0$.
\end{remark}
We know that there is no proper biharmonic curve in $\r^n$ (see e.g., \cite{Ou2}). For some results on biharmonic curves, we refer the readers to \cite{Ba, CMO,CMOP,COP,Ji2,Di,Ou2} and the references therein. At the end of this section, we will try to construct some explicit examples of
proper $f$-biharmonic curves in a space form.
\begin{proposition}\label{CS0}
(1)\; If $\rho'(s)\neq\pm1$ solves the $ODE$  \begin{equation}\label{G1}
\begin{array}{lll}
\rho''+(1-\rho'^2)\tan\rho+\frac{\sqrt{1-\rho'^2}}{C_4\cos(2s)+C_5\sin(2s)+\sqrt{1+C_4^2+C_5^2}}=0,
\end{array}
\end{equation}
then the curve $\gamma:(a,b)\to (S^2, d\rho^2+\cos^2\rho d\phi^2)$  with $\gamma(s)=(\rho(s),\; \int\frac{\sqrt{1-\rho'^2}}{\cos\rho}ds )$
followed by a totally geodesic embedding  $\psi:S^2\to S^n$ is a proper $f$-biharmonic curve $\tilde{\gamma}=\psi\circ\gamma:(a,b)\to S^n$  for $f=c_1\left(C_4\cos(2s)+C_5\sin(2s)+\sqrt{1+C_4^2+C_5^2}\right)^{3/2}$.

(2)\; If  $u'(s)\neq\pm1$ solves the $ODE$  \begin{equation}\label{G2}
\begin{array}{lll}
u''(s)-(1-u'^2)+ \frac{\sqrt{1-u'^2}}{C_1e^{2s}+C_2e^{-2s}+\sqrt{4C_1C_2-1}}=0,
\end{array}
\end{equation}
then the curve $\gamma:(a,b)\to (H^2, du^2+e^{2u} dv^2)$  with $\gamma(s)=(u(s),\int\frac{\sqrt{1-u'^2}}{e^u}ds)$ followed by a totally geodesic embedding  $\psi:H^2\to H^n$ is a proper $f$-biharmonic curve $\tilde{\gamma}=\psi\circ\gamma:(a,b)\to H^n$  for $f=c_1\left(C_1e^{2s}+C_2e^{-2s}+\sqrt{4C_1C_2-1}\right)^{3/2}$.

(3)\; If  $x'(s)\neq\pm1$ solves the $ODE$  \begin{equation}\label{G3}
\begin{array}{lll}
x''(s)+\frac{4C_3\sqrt{1-x'^2}}{16+C_3^2(s+C_4)^2}=0,
\end{array}
\end{equation}
then the planar curve $\gamma:(a,b)\to (\r^2, dx^2+ dy^2)$  with $\gamma(s)=(x(s),\int\sqrt{1-x'^2}ds)$ followed by a totally geodesic embedding  $\psi:\r^2\to \r^n$ is a proper $f$-biharmonic curve $\tilde{\gamma}=\psi\circ\gamma:(a,b)\to \r^n$  for $f=c_1\left(\frac{4C_3}{16+C_3^2(s+C_4)^2}\right)^{-3/2}$.
\end{proposition}
\begin{proof}
Let $\gamma:(a,b)\to (S^2, d\rho^2+\cos^2\rho d\phi^2)$, $\gamma(s)=(\rho(s),\;\phi(s))$ be a proper $f$-biharmonic curve  parametrized by arc length. It follows from the Statement (i) of Theorem \ref{RSH}
that the curve $\gamma$ has the  geodesic curvature as
\begin{equation}\label{sno1}\notag
\begin{array}{lll}
\kappa_1(s)=\chi(s,0,1)=\frac{1}{C_4\cos(2s)+C_5\sin(2s)+\sqrt{1+C_4^2+C_5^2}}.
\end{array}
\end{equation}
Suppose that $\theta$ is the angle between the curve $\gamma$ and $\rho$-curves.  Therefore,
applying the Liouville's formula for the geodesic curvature of curves on the surface $S^2$  in  $\r^3$ yields
\begin{equation}\label{sno6}
\begin{array}{lll}
\theta'(s)-\sin\theta\tan\rho(s)=\kappa_1(s).
\end{array}
\end{equation}
Since $\rho'=\cos\theta$ and $\phi'\cos\rho=\sin\theta$,  then (\ref{sno6}) implies that
 \begin{equation}\notag
\begin{array}{lll}
\rho''+(1-\rho'^2)\tan\rho+\kappa_1\sqrt{1-\rho'^2}=0,
\end{array}
\end{equation}
and hence $\phi(s)=\int\frac{\sqrt{1-\rho'^2}}{\cos\rho}ds$.  Note that $\kappa_1(s)\neq0$, i.e., $\sin\theta\neq0$ and hence $\cos\theta\neq\pm1$, then $\rho'\neq\pm 1$. Hence, using  Statement (i) of Theorem \ref{RSH} and Remark \ref{002}, we obtain the statement (1).

 Similar to  the statement (1), we can get  the statement (2) and (3).

 Summarizing all the above results, we obtain the proposition.
\end{proof}

 We can apply Statement (3) of Proposition \ref{CS0} to construct infinitely many examples of  proper $f$-biharmonic curves  in $\r^n$.
\begin{example}\label{cf1}
 A family of planar curves $\tilde{\gamma}(s)=(\frac{4\ln|\sqrt{16+C_3^2s^2}+C_3s|}{C_3}-\frac{4\ln C_3}{C_3},\;\frac{\sqrt{16+C_3^2s^2}}{C_3},\\0,\ldots,\;0)$\; in $\r^n$\; are proper $f$-biharmonic curves\;for $f=c_1(\frac{4C_3}{16+C_3^2s^2})^{-3/2}$,\;where $c_1$ and \;$C_3$ are positive constants.
Moreover, the curvature $\kappa_1=\frac{4C_3}{16+C_3^2s^2}$ and the torsion $\kappa_2=0$.  The  example recovers the family
of proper $f$-biharmonic curves found in \cite{Ou10}.

In fact, we solve (\ref{G3}) to obtain
 the general solutions as
\begin{equation}\label{GG3}
\begin{array}{lll}
x=-\frac{\cos(4C_3a_2)\sqrt{16+C_3^2(s^2+2C_4s+C_4^2)}}{C_3}-\frac{4\sin(4C_3a_2)\ln|C_3s+C_3C_4+\sqrt{16+C_3^2(s^2+2C_4s+C_4^2)}|}{C_3}+a_1,
\end{array}
\end{equation}
where $a_1$ and $a_2$ are constants.  By taking $C_3>0$,\;$4C_3a_2=-\pi/2$, $a_1=-\frac{4\ln C_3}{C_3}$ and $C_4=0$, one concludes that $x(s)=\frac{4\ln|\sqrt{16+C_3^2s^2}+C_3s|}{C_3}-\frac{4\ln C_3}{C_3}$  and hence $y(s)=\int\sqrt{1-x'^2}ds=\frac{\sqrt{16+C_3^2s^2}}{C_3}$.
Clearly, the curvature $\kappa_1=\frac{4C_3}{\sqrt{16+C_3^2s^2}}$ and the torsion $\kappa_2=0$ in this case.  Therefore, we apply  Statement (3) of Proposition \ref{CS0} to obtain the example.
\end{example}

We will give a family of  proper $f$-biharmonic general helixes  in $\r^3$.
\begin{proposition}\label{Hl}
If $x'(s)\neq\pm\sin\omega$ solves the $ODE$
\begin{equation}\label{hl}
\begin{array}{lll}
x''(s)+\frac{\sqrt{\sin^2\omega-x'^2}}{\sin\omega}\frac{4C_3}{16(1+c_3^2)+C_3^2(s+C_4)^2}=0,
\end{array}
\end{equation}
where $\sin\omega\cos\omega\neq0$ and $c_3=\frac{\cos\omega}{\sin\omega}$ are constants, then a general helix $\gamma:(a,b)\to\r^3$  with $\gamma(s)=(x(s),\int\sqrt{\sin^2\omega-x'^2}ds ,s\cos\omega)$ followed by a totally geodesic embedding  $\psi:\r^3\to \r^n$ is a proper $f$-biharmonic curve $\tilde{\gamma}=\psi\circ\gamma:(a,b)\to \r^n$  for $f=c_1(\frac{4C_3}{16(1+c_3^2)+C_3^2(s+C_4)^2})^{-3/2}$.
\end{proposition}
\begin{proof}
Let a general helix $\gamma:(a,b)\to\r^3$  with $\gamma(s)=(x(s),y(s),s\cos\omega)$  be a proper $f$-biharmonic curve parametrized by arclength, where constant $\cos\omega\neq0$. By Corollary \ref{Th3}, the helix has the curvature $\kappa_1(s)= \frac{4C_1}{16(1+c_3^2)+C_1^2(s+2C_2)^2}\neq0$ and the torsion $\kappa_2=c_3\kappa_1(s)\neq0$. Since $\gamma'(s)=(x'(s),y'(s),\cos\omega)$ and $|\gamma'(s)|^2=x'^2(s)+y'^2+\cos^2\omega=1$, we may assume that
 $x'(s)=\sin\omega\cos\theta(s),\;y'=\sin\omega\sin\theta$. Therefore, we have
\begin{equation}\label{hl1}
\begin{array}{lll}
\gamma'=(\sin\omega\cos\theta,\sin\omega\sin\theta,\cos\omega),\;
\gamma''=\sin\omega\theta'(-\sin\theta,\cos\theta,0),\\
\gamma'''=\sin\omega\theta''(-\sin\theta,\cos\theta,0)-\sin\omega\theta'^2(\cos\theta,\sin\theta,0).
\end{array}
\end{equation}
It follows that
\begin{equation}\label{hl2}
\begin{array}{lll}
\kappa_1(s)=\sin\omega\theta',\;\kappa_2=\frac{(\gamma',\gamma'',\gamma''')}{\kappa_1^2}=\cos\omega\theta',\;
{\rm and\; hence}\; c_3=\frac{\cos\omega}{\sin\omega}.
\end{array}
\end{equation}
 Combining these, together with  $\kappa_1(s)\neq0$ and $\kappa_2\neq0$, we have $\sin\omega\cos\omega\neq0$ and $x'\neq\pm\sin\omega$.
A direct computation gives $x''(s)=-\sin\omega\theta'\sin\theta=-\theta'y'=-\theta'\sqrt{\sin^2\omega-x'^2}$. Substituting this into the 1st equation of (\ref{hl2})
we obtain (\ref{hl}). Clearly, $y=\int\sqrt{\sin^2\omega-x'^2}ds$. From these and using Remark \ref{002}, the proposition follows.
\end{proof}

We will look for some special solutions of (\ref{hl}).
\begin{example}\label{cf2}
 A family of  helixes $\gamma(s)=(\frac{4\ln|\sqrt{64+2C_3^2s^2}+\sqrt{2}C_3s|}{C_3}-\frac{4\ln C_3}{C_3},\;\frac{\sqrt{64+2C_3^2s^2}}{C_3},\frac{\sqrt{2}}{2}s)$\; in $\r^3$\; with the curvature $\kappa_1=\frac{4C_3}{32+C_3^2s^2}$ and the torsion $\kappa_2=\frac{4C_3}{32+C_3^2s^2}$ are proper $f$-biharmonic curves\;for $f=c_1(\frac{4C_3}{32+C_3^2s^2})^{-3/2}$,\;where $c_1$ and \;$C_3$ are positive constants.
Moreover, a family of curves $\bar{\gamma}(s)=(\frac{4\ln|\sqrt{64+2C_3^2s^2}+\sqrt{2}C_3s|}{C_3}-\frac{4\ln C_3}{C_3},\;\frac{\sqrt{64+2C_3^2s^2}}{C_3},\\ \frac{\sqrt{2}}{2}s,\;0,\ldots,0)$\; in $\r^n$\; are also proper $f$-biharmonic curves\;for $f=c_1(\frac{4C_3}{32+C_3^2s^2})^{-3/2}$.

In fact, by  taking $\sin\omega=\cos\omega=\frac{\sqrt{2}}{2}$,\;$c_3=1$  and $C_4=0$, (\ref{hl}) becomes
\begin{equation}\label{hll}
\begin{array}{lll}
x''(s)+\frac{\sqrt{\frac{1}{2}-x'^2}}{\frac{\sqrt{2}}{2}}\frac{4C_3}{32+C_3^2s^2}=0,
\end{array}
\end{equation}
which is solved by
\begin{equation}\label{G3}\notag
\begin{array}{lll}
x=-\frac{\cos(4\sqrt{2}C_3a_2)\sqrt{64+2C_3^2s^2}}{2C_3}-\frac{4\sin(4\sqrt{2}C_3a_2)\ln|\sqrt{2}C_3s+\sqrt{64+2C_3^2s^2}|}{C_3}+a_1,
\end{array}
\end{equation}
where $a_1$ and $a_2$ are constants. Note that $y(s)=\int \sqrt{1-x'^2}ds$. Furthermore,
 taking $C_3>0$,\;$4\sqrt{2}C_3a_2=-\pi/2$ and $a_1=-\frac{4\ln C_3}{C_3}$, one finds that $x(s)=\frac{4\ln|\sqrt{64+2C_3^2s^2}+\sqrt{2}C_3s|}{C_3}-\frac{4\ln C_3}{C_3}$  and hence $y(s)=\int\sqrt{1-x'^2}ds=\frac{\sqrt{64+2C_3^2s^2}}{C_3}$.
Clearly, the curvature $\kappa_1=\frac{4C_3}{\sqrt{32+C_3^2s^2}}$ and the torsion $\kappa_2=\frac{4C_3}{\sqrt{32+C_3^2s^2}}$ in this case.  Thus, we apply  Proposition \ref{Hl} to get the example.
\end{example}

\subsection{ $f$-Biharmonic  isometric immersions and biharmonic  conformal immersions of a developable surface into $\r^3$ }
The  equation for $f$-biharmonic hypersurfaces in
a space form can be stated as
\begin{lemma}$($ see e.g.,\cite{Ou10, Ou11}$)$\label{MTH}
A hypersurface $\varphi:M^{m}\to N^{m+1}(C)$ in a space forms of constant sectional curvature $C$ with mean curvature vector
field $\eta=H\xi$ is $f$-biharmonic if and only if the function $H$
 statisfies the following equation
\begin{equation}\label{fb}
\begin{cases}
\Delta(fH)-(fH)[|A|^{2}-
mC]=0,\\
 A\,({\rm grad}\,(fH)) +\frac{m}{2} (fH){\rm grad}\, H
=0,
\end{cases}
\end{equation}
where ${\rm Ric}^N : T_qN\to T_qN$ denotes the Ricci
operator of the ambient space defined by $\langle {\rm Ric}^N\,
(Z), W\rangle={\rm Ric}^N (Z, W)$ and  $A$ is the shape operator
of the hypersurface with respect to the unit normal vector field $\xi$.
\end{lemma}

We are ready to give a  characterization of proper $f$-biharmonic cylinders in $\r^3$.
\begin{theorem}\label{cy}
Let $X:(M^2,g)\to(\r^3,h={\rm d}x^2+{\rm d}y^2+{\rm d}z^2)$,\; $X(s,v)=a(s)+vb$  be  a cylinder with
the mean curvature function $H$, where $b$  is a unit constant vector in $\r^3$ and $a(s)$ is  an immersed regular curve in $\r^3$  parametrized by arc length  with the geodesic curvature $\kappa_1(s)$  satisfying  $h(a', b)=0$. Then, $X$ is  proper $f$-biharmonic if and only if one of the following cases happens:\\
$(1)$  $\kappa_1$ is a nonzero constant,  and the directrix $a(s)$ of the  cylinder is (a part of) a circle with radius $\frac{1}{\sqrt{|\kappa_1|}}$. Moreover, the  mean curvature $H=\frac{\kappa_1}{2}$ and the function $f=d_1e^{\kappa_1v }+d_2e^{-\kappa_1 v}$, where $d_1$ and $d_2$ are constants.\\or,\\
$(2)$  $\kappa_1(s)$ is  nonconstant, and the directrix $a(s)$ of the  cylinder has the geodesic curvature $\kappa_1(s)=\kappa_N(s,K)$ given by (\ref{RH00}) and the geodesic torsion $\kappa_2=0$, where $K$ is a constant. Furthermore, the  mean curvature $H=\frac{\kappa_N(s,K)}{2}$ and the function $f= \psi(v,K)\kappa_N^{-3/2}(s,K)$, where $\psi(v,K)$ given by (\ref{cy8}).
 \end{theorem}
\begin{proof}
A straightforward computation gives
\begin{equation}\label{cy1}\notag
\begin{array}{lll}
X_s=a'(s),\; X_v=b,\; X_{ss}=a''(s),\; X_{sv}=0,\;X_{vv}=0,\\
N=\frac{X_s\times X_v}{|X_s\times X_v|}=a'\times b,\;\; |a''(s)|=\kappa(s),
\end{array}
\end{equation}
where $\kappa$ is the curvature of the curve $a(s)$.
By a further computation, the first  fundamental form $I$  and the second fundamental form  $II$ of the cylinder given by
\begin{equation}\label{cy3}\notag
\begin{array}{lll}
I=g={\rm d}s^2+{\rm d}v^2,\;\; II= \kappa_N(s){\rm d}s^2,
\end{array}
\end{equation}
where $\kappa_N$  is  the normal curvature of the curve $a(s)$.
It follows that $\{e_1=\frac{\partial}{\partial s}$,\; $e_2=\frac{\partial}{\partial v},\;\xi=N\}$  forms an orthonormal frame  adapted to the surface
with $\xi$ being normal vector field, $A(e_1)=II(e_1,e_1)e_1+II(e_1,e_2)e_2=\kappa_N e_1$,  $A(e_2)=II(e_2,e_1)e_1+II(e_2,e_2)e_2=0$ and  $H=\frac{\kappa_N}{2}$. Substituting these into the $f$-biharmonic equation (\ref{fb}) with $m=2$ and $C=0$ we have
\begin{equation}\label{cy4}
\begin{cases}
\kappa_N^2\frac{\partial \ln f}{\partial s}=-\frac{3}{2}\kappa_N'(s)\kappa_N,\\
\kappa_N''(s)-\kappa_N^3+\kappa_N[\frac{\partial^2\ln f}{\partial s^2}+\frac{\partial^2\ln f}{\partial v^2}+(\frac{\partial\ln f}{\partial s})^2+(\frac{\partial\ln f}{\partial v})^2]+2\frac{\partial\ln f}{\partial s}\kappa_N'(s)=0.
\end{cases}
\end{equation}
If $\kappa_N=0$, i.e., $H=0$, then the surface is harmonic and hence biharmonic, not proper $f$-biharmonic. From now on, we assume that $\kappa_N\neq0$.
If $\kappa_N$ is a nonzero constant,  then   (\ref{cy4}) turns into
\begin{equation}\label{cy5}
\begin{array}{lll}
\frac{\partial \ln f}{\partial s}=0,\;\;-\kappa_N^2+\frac{\partial^2\ln f}{\partial v^2}+(\frac{\partial\ln f}{\partial v})^2=0,
\end{array}
\end{equation}
which is solved by $ f=d_1e^{v\kappa_N }+d_2e^{-v\kappa_N }$, where $d_1$ and $d_2$ are constants. Note that the curve $a(s)$ can be viewed as (a part of) a circle of  radius $\frac{1}{\sqrt{|\kappa_N|}}$.\\
We now assume that $\kappa$ is nonconstant and apply the first equation of (\ref{cy4}) to have
\begin{equation}\label{cy5}
\ln f=-\frac{3}{2}\ln |\kappa_N(s)| +\ln \psi(v),\;{\rm and\;hence}\;f=\psi(v)\kappa_N^{-3/2}(s),
\end{equation}
where $\psi(v)$ is a positive function.
Substituting (\ref{cy5}) into  the second equation of (\ref{cy4}) and simplifying the resulting equation we obtain
\begin{equation}\label{cy05}
-2\kappa_N\kappa_N''(s)+3\kappa_N'^2(s)-4\kappa_N^4+4\kappa_N^2(s)\frac{\psi''(v)}{\psi(v)}=0.
\end{equation}
The equation (\ref{cy05}) implies that for any $s,v$, we have
\begin{equation}\label{cy6}
\begin{array}{lll}
3\kappa_N'^2(s)-2\kappa_N\kappa_N''(s)=4\kappa_N^2(\kappa_N^2-K)\;\;{\rm and}\;\;
\frac{\psi''(v)}{\psi(v)}=K,
\end{array}
\end{equation}
 where $K$ is a constant.
The 2nd equation of (\ref{cy6}) is equivalent to $\psi''(v)=K\psi(v)$  solved by
 \begin{equation}\label{cy8}
\psi(v)= \psi(v,K)=\begin{cases}
d_5\sin(\sqrt{-K}\;v)+d_6\cos(\sqrt{-K}\;v),\;\;\; {\rm for}\; K<0,\\
d_3v+d_4,\;\;\;\;\;\;\;\;\;\;\;\;\;\;\;\;\; {\rm for}\; K=0,\\
d_1e^{\sqrt{K}\;v}+d_2e^{-\sqrt{K}\;v},\;\;\;\;\;\;\;\; {\rm for}\; K>0,
\end{cases}
\end{equation}
 where $d_i$, $i=1,2,\ldots,6,$  are constants.\\
An interesting thing is  that the 1st equation of (\ref{cy6}) happens to be the type of $ODE$ (\ref{ode}) in Proposition \ref{ODE} when $A=1$ and $C=K$, which is solved by
\begin{equation}\label{RH00}
\kappa_N(s)=\kappa_N(s,K)=\chi(s,0,K)=\begin{cases}
\frac{1}{C_1e^{2\sqrt{-K}\;s}+C_2e^{-2\sqrt{-K}\;s}+\sqrt{4C_1C_2+\frac{1}{K}}},\;\;\;\;\;\;\;\;{\rm for}\; K<0,\\
\frac{4C_3}{16+C_3^2(s+C_4)^2},\;\;\;\;\;\;\;\;\;\;\;\;\;\;\;\;\;\;\;\;\;\;\;\;\;\;\;\;\;\;{\rm for}\; K=0,\\
\frac{1}{C_4\cos(2\sqrt{K}\;s)+C_5\sin(2\sqrt{K}\;s)+ \sqrt{\frac{1}{K}+(C_4^2+C_5^2)}},\;{\rm for}\; K>0,
\end{cases}
\end{equation}
where $C_i$ are constants, $i=1,2,\ldots,6$.

It is not difficult to check that  $\nabla_{e_1}e_1=\kappa_Ne_3,\;\nabla_{e_1}e_3=-\kappa_Ne_1$ and\;$\nabla_{e_1}(-e_2)=0$, and hence $\{e_1,\;e_3=N,\;-e_2\}$ forms the Frenet frame along the curve $a(s)$. This implies that the curve $a(s)$ has the geodesic curvature $\kappa_1=\kappa_N$ and the geodesic torsion $\kappa_2=0$.

Summarizing all the above results, we complete the proof of the theorem.
\end{proof}
\begin{remark}\label{a1}
 Theorem \ref{cy}  recovers  the classification result of proper $f$-biharmonc cylinders with constant mean curvature  in \cite{Ou10}.
\end{remark}
\begin{corollary}\label{cyo}
Let $X:(M^2,g)\to\r^3$ be  a cylinder with\; $X(s,v)=a(s)+vb$ and
nonconstant mean curvature function $H$,  where $b$  is a unit constant vector in $\r^3$ and $a(s)$ is  an immersed regular curve in $\r^3$  parametrized by arc length  with the geodesic curvature $\kappa_1(s)$ satisfying  $h(a', b)=0$. Then, $X$ is  proper $f$-biharmonic if and only if the directrix $a(s)$
is a proper $\bar{f}$-biharmonic curve in a 2-space form $N^2(K)\subset\r^3$ with constant Gauss curvature $K$ for $\bar{f}=c_1\kappa_1^{-3/2}$, where $c_1>0$ is a constant and  $\kappa_1(s)=\kappa_N(s,K)$ given by (\ref{RH00}). Moreover, the function $f= \frac{\psi(v,K)}{c_1}\bar{f}=\psi(v,K)\kappa_N^{-3/2}(s,K)$, where $\psi(v,K)$ given by (\ref{cy8}).
 \end{corollary}
\begin{proof}
From the proof of Theorem \ref{cy} and  Statement (i) of Proposition \ref{fc0}, together with the assumption that $H$ is nonconstant, one sees that the directrix $a(s)$ on the surface has the geodesic torsion $\kappa_2=0$ and  nonconstant geodesic curvature $\kappa_1=\kappa_N=2H$  solving  the 1st equation of (\ref{cy6}) which is a proper $\bar{f}$-biharmonic curve equation in a space form of constant sectional curvature $K$ for $\bar{f}=c_1\kappa_1^{-3/2}$, where constant $c_1>0$. This implies that the  directrix $a(s)$ can be viewed as a proper $\bar{f}$-biharmonic curve in a 2-space form $N^2(K)$ of constant Gauss curvature $K$. Note that  $f= \frac{\psi(v,K)}{c_1}\bar{f}$, where $\psi(v,K)$ given by (\ref{cy8}).
\end{proof}

 \begin{proposition}\label{ccy1}
There is neither a proper $f$-biharmonic cone nor tangent surface in $\r^3$.
 \end{proposition}
\begin{proof}
Let $b(s)$ be a curve on a unit sphere parametrized by arc length (i.e., a spherical curve).  Consider  a cone $X:(M^2,g)\to(\r^3,h={\rm d}x^2+{\rm d}y^2+{\rm d}z^2)$  into  $\r^3$  with $X(s,v)=a+vb(s)$ , where $a$ is a constant vector. A simple computation, we obtain the first  and the second fundamental forms $I$ and $II$ of the cone as
\begin{equation}\label{pz3}\notag
\begin{array}{lll}
I=g=v^2{\rm d}s^2+{\rm d}v^2,\;\;
II=vw(s){\rm d}s^2,
\end{array}
\end{equation}
where $w(s)=b''\cdot(b'\times b)$.
One can check that $\{e_1=\frac{1}{v}\frac{\partial}{\partial s}$,\;$e_2=\frac{\partial}{\partial v},\;\xi=b'\times b\}$ constitutes an orthonormal frame
adapted to the surface with $\xi$ being normal to the surface. By a straightforward computation, we have  $A(e_1)=II(e_1,e_1)e_1+II(e_1,e_2)e_2=\frac{w}{v}e_1$, $A(e_2)=II(e_2,e_1)e_1+II(e_2,e_2)e_2=0$ and  $H=\frac{w(s)}{2v}$.
Substituting these into the second equation of $(\ref{fb})$ and simplifying the resulting equation yields
\begin{equation}\label{pz4}
H[3e_1(H)+2He_1(\ln f)]=0,\;He_2(H)=H\frac{\partial}{\partial v}\left(\frac{w(s)}{2v}\right)=0.
\end{equation}
The 2nd equation of (\ref{pz4}) implies that
$-\frac{H^2}{ 2v^2}=0$ and hence $H=0$, i.e., the cone has to be minimal. It follows that any $f$-biharmonic cone in $\r^3$ has to be harmonic and hence biharmonic, but not proper $f$-biharmonic.

For a tangent surface in  $\r^3$, let $a(s)$ be a curve parametrized by arclength and let $X:(M^2,g)\to \r^3$ be a tangent surface  with $X(s,v)=a(s)+va'(s)$. By simple computation, the first  and the second fundamental forms $I$ and $II$ of the tangent surface given by respectively
\begin{equation}\label{ts1}
\begin{array}{lll}
I=g=(1+v^2\kappa^2){\rm d}s^2+2{\rm d}s{\rm d}v+{\rm d}v^2,\;\;
II=-v\kappa\tau{\rm d}s^2,
\end{array}
\end{equation}
where $\kappa=\kappa(s)$ and \;$\tau=\tau(s)$ are the curvature and torsion of   $a(s)$ respectively.

One can easily check  that $\{\varepsilon_1=\frac{1}{\sqrt{1+v^2\kappa^2}}\frac{\partial}{\partial s},\;\varepsilon_2=\frac{1}{v\kappa\sqrt{(1+v^2\kappa^2}}\frac{\partial}{\partial s}-\frac{\sqrt{1+v^2\kappa^2}}{v\kappa}\frac{\partial}{\partial v},\\ \xi=\frac{a''\times a'}{\kappa}\}$ forms an orthonormal frame adapted to the surface with the normal vector field $\xi$.  Suppose that $A(\varepsilon_1)=\kappa_1e_1+\kappa_{12}\varepsilon_2$ and $A(\varepsilon_2)=\kappa_{21}\varepsilon_1+\kappa_{2}\varepsilon_2$. A direct computation gives $\kappa_1=II(\varepsilon_1,\varepsilon_1)=-\frac{v\kappa\tau}{1+v^2\kappa^2}$, $\kappa_2=II(\varepsilon_2, \varepsilon_2)=-\frac{\tau}{v\kappa(1+v^2\kappa^2)}$ and $\kappa_{12}=\kappa_{21}=II(e_1,e_2)=-\frac{\tau}{1+v^2\kappa^2}$.  It follows that the mean curvature
 $H=\frac{\kappa_1+\kappa_2}{2}=-\frac{\tau}{2v\kappa}$ and $\kappa_1\kappa_2=\frac{\tau^2}{(1+v^2\kappa^2)^2}$. Note that $H=0$ is equivalent to $\tau=0$, then the surface is not proper $f$-biharmonic in this case.

 Hereafter, assume that $H\neq0$. Let  $\{e_1=\cos\theta\varepsilon_1+\sin\theta\varepsilon_2,\;e_2=-\sin\theta\varepsilon_1+\cos\theta\varepsilon_2,\;\xi\}$
 be another orthonormal frame adapted to the surface with the normal vector field  $\xi$  such that $A(e_1)=\lambda_1e_1$ and $A(e_2)=\lambda_{2}e_2$,
 where $\theta$ is the angle between $e_1$ and $\varepsilon_1$. Since the tangent surface is flat,  it follows from  the Gauss equation of the tangent surface that $\lambda_1\lambda_2=0$. Without a loss of generality, we may suppose that $\lambda_2=0$ and hence $\lambda_1=2H\neq0$. We can easily conclude that
 $\kappa_1=2H\cos^2\theta,\;\kappa_2=2H\sin^2\theta,\;\kappa_{12}=\tau=2H\sin\theta\cos\theta,\;
\cos^2\theta=\frac{v^2\kappa^2}{1+v^2\kappa^2}$\;and\; $\sin^2\theta=\frac{1}{1+v^2\kappa^2}.$
 If the tangent surface is $f$-biharmonic, by (\ref{fb}) with $m=2$ and $C=0$, we have
 \begin{equation}\label{ts4}
\begin{array}{lll}
\Delta(fH)-4fH^3=0,\;2He_1(fH)+fHe_1(H)=0,\;fHe_2(H)=0,
\end{array}
\end{equation}
which implies that $e_2(H)=0$. Therefore,  we get $[-\sin\theta\varepsilon_1+\cos\theta\varepsilon_2](H)=0$, i.e., $\sin^2\theta(\varepsilon_1(H))^2-\cos^2\theta(\varepsilon_2(H))^2=0$. A further computation gives
\begin{equation}\label{ts5}
\begin{array}{lll}
0=\sin^2\theta[\varepsilon_1(-\frac{\tau}{2v\kappa})]^2
-\cos^2\theta[\varepsilon_2(-\frac{\tau}{2v\kappa})]^2\\
=\frac{1}{1+v^2\kappa^2}\left(\frac{1}{\sqrt{1+v^2\kappa^2(s)}}\frac{\partial}{\partial s}(\frac{\tau}{2v\kappa})\right)^2
-\frac{v^2\kappa^2}{1+v^2\kappa^2}\left(\frac{1}{v\kappa\sqrt{(1+v^2\kappa^2}}\frac{\partial}{\partial s}(\frac{\tau}{2v\kappa})-\frac{\sqrt{1+v^2\kappa^2}}{v\kappa}\frac{\partial}{\partial v}(\frac{\tau}{2v\kappa})\right)^2\\
=-\frac{(\frac{\tau}{\kappa})^2}{4v^4}-\frac{[(\frac{\tau}{\kappa})^2]'(s)}{4v^3(1+v^2\kappa^2)}.
\end{array}
\end{equation}
 It is easy to check that (\ref{ts5}) implies that
 \begin{equation}\label{ts6}
\begin{array}{lll}
\tau^2(s)\;v^2+[(\frac{\tau}{\kappa})^2]'(s)\;v+(\frac{\tau}{\kappa})^2(s)=0,
\end{array}
\end{equation}
which, together with any $s,\;v$, implies that $\tau=0$. This contradicts to the assumption that  $\tau\neq0$.
 Combining these, any $f$-biharmonic tangent surface in $\r^3$ is  harmonic, but not proper $f$-biharmonic.

Summarizing all the above results, the proposition follows.
\end{proof}
Applying Theorem \ref{cy} and Proposition \ref{ccy1}, we have
\begin{corollary}\label{zh}
A proper $f$-biharmonic  developable surface in $\r^3$  exists only in the case when the surface is either a circle cylinder or a cylinder
containing a  directrix  with nonconstant geodesic curvature $\kappa_1(s)=\kappa_N(s,K)$ given by (\ref{RH00}) and the geodesic torsion  $\kappa_2(s)=0$.
\end{corollary}
We will construct a family of  cylinders  whose directrices  are  proper $\bar{f}$-biharmonic curves in  $\r^2$ for $\bar{f}=c_1\kappa_N^{-3/2}(s,0)$, which are   proper $f$-biharmonic cylinders for $f= \frac{\psi(v,0)}{c_1}\bar{f}$,\; where $\psi(v,0)$ and $\kappa_N^{-3/2}(s,0)$ given by (\ref{cy8}) and (\ref{RH00}) respectively.
\begin{example}\label{e2}
 A family of cylinders $X:(M^2,g)\to\r^3$, $X(s,v)=(\frac{4\ln|\sqrt{16+C_3^2s^2}+C_3s|}{C_3}-\frac{4\ln C_3}{C_3},\;\frac{\sqrt{16+C_3^2s^2}}{C_3},\; v)$ are proper $f$-biharmonic  for $f=\frac{c_1(d_3v+d_4)(16+C_3^2s^2)^{3/2}}{(4C_3)^{3/2}}$,\;where $d_3$, $d_4$, $c_1$ and \;$C_3$ are positive constants. This family of cylinders followed by a totally geodesic embedding $\phi:\r^3\to \r^n$  are  proper $f$-biharmonic submanifolds $\phi\circ X:(M^2,g)\to\r^n$.

In fact, it is easy to check that $X(s,v)=a(s)+vb$, where  $a(s)=(\frac{4\ln|\sqrt{16+C_3^2s^2}+C_3s|}{C_3}-\frac{4\ln C_3}{C_3},\;\frac{\sqrt{16+C_3^2s^2}}{C_3},\;0)$ are a family of  curves in $\r^2\subset\r^3$  and  $b=(0,0,1)$ is a unit  constant vector in $\r^3$. Clearly, $|a'|^2=1$ and $h(a',b)=0$.
Therefore, by Example \ref{cf1}, any  curve $a(s)$ is a proper $\bar{f}$-biharmonic curve  in $\r^2\subset\r^3$ for $\bar{f}=c_1(\frac{4C_3}{16+C_3^2s^2})^{-3/2}$, where $c_1>0$ is a constant. Then, applying Corollary \ref{cyo}, we obtain the example.
\end{example}
We know from \cite{Ou10} that  a map $\varphi: (M^2, g)\to (N^n, h)$ is an $f$-biharmonic map if and only if the conformal immersion $\varphi:(M^2, f^{-1}g)\to (N^n, h)$ is a biharmonic map. Therefore,
applying  Theorem \ref{cy} and Proposition \ref{ccy1}, we  have the following corollaries as follows
\begin{corollary}\label{ccy}
Under the same assumptions as in Theorem \ref{cy}, then we have\\
$(1)$  If $\kappa_1$ is a constant and $f=d_1e^{\kappa_1 v}+d_2e^{-\kappa_1 v}>0$, where $d_1$ and $d_2$ are  constants, then the conformal immersion $X:(M^2,f^{-1}g)\to\r^3$ with $X(s,v)=a(s)+vb$ is proper biharmonic,\;or\\
$(2)$ If $\kappa_1(s)=\kappa_N(s,K)$ and $\psi(v,K)$ given by (\ref{RH00}) and (\ref{cy8}) respectively, and $f= \psi(v,K)\kappa_N^{-3/2}(s,K)$, where $K$ is a constant, then the conformal immersion $X:(M^2,f^{-1}g)\to\r^3$ with $X(s,v)=a(s)+vb$ is proper biharmonic.
 \end{corollary}

\begin{corollary}\label{ccy2}
We have\\
(1) There can not exist  a proper biharmonic conformal immersion of a cone or a tangent surface  into  $\r^3$.\\
(2) A conformal immersion of a developable surface  into  $\r^3$  is proper biharmonic if and only if the surface is either a circle cylinder or a cylinder
containing   a directrix   with  nonconstant geodesic curvature $\kappa_1(s)=\kappa_N(s,K)$ given by (\ref{RH00}) and the torsion  $\kappa_2(s)=0$.
 \end{corollary}
\begin{remark}\label{a2}
Corollary \ref{ccy2} recovers the result  in \cite{Ou9}  stated that there is no a proper biharmonic conformal immersion of a  circular cone into $\r^3$.
\end{remark}
Adopting  the same notations as in Example \ref{e2} and applying Example \ref{e2} and Corollary \ref{ccy}, we have
\begin{example}\label{ex3}
The conformal immersions of a family of cylinders $X:(M^2,f^{-1}g)\to\r^3$, $X(s,v)=(\frac{4\ln|\sqrt{16+C_3^2s^2}+C_3s|}{C_3}-\frac{4\ln C_3}{C_3},\;\frac{\sqrt{16+C_3^2s^2}}{C_3},\; v)$ are proper biharmonic  with  $f=\frac{c_1(d_3v+d_4)(16+C_3^2s^2)^{3/2}}{(4C_3)^{3/2}}$.
\end{example}
\section{ $f$-biharmonic Riemannian submersions from 3-manifolds}
In this section, we will study $f$-biharmonic Riemannian submersions from 3-manifolds by using the integrability data of an adapted frame of the Riemannian submersion. We also construct a family of proper $f$-biharmonic Riemannian submersions from  Sol space.

Recall that an orthonormal frame $\{e_1, e_2,e_3\}$ is adapted to a Riemannian
submersion $\pi:( M^3 , g)\to (N^2,h)$
with $e_3$ being vertical, then we have (see \cite{WO})
\begin{equation}\label{RC0}
\begin{array}{lll}
[e_1,e_3]=\kappa_1e_3,\;
[e_2,e_3]=\kappa_2e_3,\;
[e_1,e_2]=f_1 e_1+f_2e_2-2\sigma e_3,\\
\nabla_{e_{1}} e_{1}=-f_1e_2,\;\;\nabla_{e_{1}} e_{2}=f_1
e_1-\sigma e_{3},\;\;\nabla_{e_{1}} e_{3}=\sigma
e_{2},\\  \nabla_{e_{2}} e_{1}=-f_2 e_{2}+\sigma
e_3,\;\;\nabla_{e_{2}} e_{2}=f_2 e_{1}, \;\;\nabla_{e_{2}}
e_{3}=-\sigma e_{1},\\  \nabla_{e_{3}}
e_{1}=-\kappa_1e_{3}+\sigma e_{2}, \nabla_{e_{3}} e_{2}= -\sigma
e_{1}-\kappa_2 e_3, \nabla_{e_{3}} e_{3}=\kappa_1 e_{1}+\kappa_2 e_2,
\end{array}
\end{equation}
where $f_1,\;f_2,\;\kappa_1,\;\kappa_2\;{\rm and}\; \sigma \in C^{\infty}(M)$ called the integrability data of the adapted frame of  $\pi$. The bitension field of $\pi$ given by
\begin{equation}\label{fR5}
\begin{array}{lll}
\tau_2(\pi)=
[-\Delta^{M}\kappa_1-f_1e_1(\kappa_2)-e_1(\kappa_2
f_1)-f_2e_2(\kappa_2)-e_2(\kappa_2 f_2)\\+\kappa_1\kappa_2 f_1
+\kappa_2^2f_2 +\kappa_1\{-K^{N}+f_{1}^{2}+f_{2}^{2}\}
]\varepsilon_1
+(-\Delta^{M}\kappa_2+f_1e_1(\kappa_1)\\+e_1(\kappa_1
f_1)+f_2e_2(\kappa_1)+e_2(\kappa_1
f_2)
-\kappa_1\kappa_2 f_2-\kappa_1^2f_1
+\kappa_2\{-K^{N}+f_{1}^{2}+f_{2}^{2}\} )\varepsilon_2,
\end{array}
\end{equation}
where $d\pi(e_i)=\varepsilon_i,\;i=1,2$.

Using an adapted frame and the associated integrability data, the $f$-biharmonicity of a Riemannian submersion  can be described as follows.
\begin{proposition}\label{fR0}
Let $\pi:(M^3,g)\to (N^2,h)$ be a Riemannian submersion
with the adapted frame $\{e_1,\; e_2, \;e_3\}$ and the integrability
data $ \{f_1, f_2, \kappa_1,\;\kappa_2,\;\sigma\}$. Then, the
Riemannian submersion $\pi$ is $f$-biharmonic if and only if
\begin{equation}\label{fR1}
\begin{cases}
-\Delta^{M}(f\kappa_1)-2\sum\limits_{i=1}^{2}f_i e_i(f\kappa_2)-f\kappa_2\sum\limits_{i=1}^{2}\left(e_i( f_i)
-\kappa_i f_i\right)+f\kappa_1\left(-K^{N}+\sum\limits_{i=1}^{2}f_{i}^{2}\right)
=0,\\
-\Delta^{M}(f\kappa_2)+2\sum\limits_{i=1}^{2}f_i e_i(f\kappa_1)+f\kappa_1\sum\limits_{i=1}^{2}(e_i( f_i)
-\kappa_i f_i)+f\kappa_2\left(-K^{N}+\sum\limits_{i=1}^{2}f_{i}^{2}\right)=0,\\
\end{cases}
\end{equation}
where
$K^{N}=R^{N}_{1212}\circ\pi=e_1(f_2)-e_2(f_1)-f_{1}^{2}-f_{2}^{2}
$ is the Gauss curvature of Riemannian manifold $(N^2,h)$.
\end{proposition}
\begin{proof}
A direct computation using (\ref{RC0}) gives
\begin{equation}\label{fR2}
\begin{array}{lll}
\tau(\pi)=\nabla^{\pi}_{e_i}d\pi(e_i)-d\pi(\nabla^{M}_{e_i}e_i)=-d\pi(\nabla^{M}_{e_3}e_3)=-\kappa_1\varepsilon_1-\kappa_2
\varepsilon_2,\\
\Delta^Mf=\sum\limits_{i=1}^3e_ie_i(f)+f_1e_2(f)-f_2e_1(f)-\kappa_1e_1(f)-\kappa_2e_2(f),\\
2\nabla^\pi_{{\rm grad}f}\tau(\pi)=-2\{\langle{\rm grad}f,{\rm grad}\kappa_1\rangle+\kappa_2f_1e_1(f)+\kappa_2f_2e_2(f)\}\varepsilon_1 \\-
2\{\langle{\rm grad}f,{\rm grad}\kappa_2\rangle-\kappa_1f_1e_1(f)-\kappa_1f_2e_2(f)\}
\varepsilon_2.
\end{array}
\end{equation}
Substituting   (\ref{fR2}) and (\ref{fR5}) into $f$-biharmonic map equation (\ref{fbeq}) and a straightforward computation gives
\begin{equation}\notag
\begin{array}{lll}
0=\tau_{2,f}(\pi)=-J^{\pi}(f\tau(\pi))\\
=f\tau^2(\pi)+(\Delta^M f)\tau(\pi)+2\nabla^\pi_{{\rm grad}f}\tau(\pi)\\
=[-\Delta^{M}(f\kappa_1)-f_1 e_1(f\kappa_2)-e_1(f\kappa_2 f_1)-f_2
e_2(f\kappa_2)-e_2(f\kappa_2 f_2)
\\+f\kappa_2(\kappa_1 f_1 +\kappa_2 f_2)
+f\kappa_1\{-K^{N}+f_{1}^{2}+f_{2}^{2}\}
]\varepsilon_1,\\
+[-\Delta^{M}(f\kappa_2)+f_1 e_1(f\kappa_1)+e_1(f\kappa_1 f_1)+f_2
e_2(f\kappa_1)+e_2(f\kappa_1 f_2)\\
-f\kappa_1(\kappa_1f_1 +\kappa_2f_2)
+f\kappa_2\{-K^{N}+f_{1}^{2}+f_{2}^{2}\}] \varepsilon_2.
\end{array}
\end{equation}
From which the proposition follows.
\end{proof}
When the integrability data $\kappa_2=0$ we have the following corollary
\begin{corollary}\label{Cor1}
Let $\pi:(M^3,g)\to(N^2,h)$ be a
Riemannian submersion with the adapted frame $\{e_1,\; e_2,
\;e_3\}$ and the integrability data $ \{f_1, f_2,
\kappa_1,\;\kappa_2,\;\sigma\}$\; with $\kappa_2=0$. Then, the Riemannian
submersion $\pi$ is $f$-biharmonic if and only if
\begin{equation}\label{cor1}
\begin{cases}
-\Delta^{M}(f\kappa_1)+f\kappa_1\{-K^{N}+f_{1}^{2}+f_{2}^{2}\}
=0,\\
2f_1 e_1(f\kappa_1)+2f_2e_2(f\kappa_1)+f\kappa_1\{e_1( f_1)+
+e_2( f_2)\}-f\kappa_1^2f_1  =0.
\end{cases}
\end{equation}
\end{corollary}
\begin{example}\label{e1}
For any positive constants  $A, B$  and  the function $f=Ae^{\sqrt{2}z}+Be^{-\sqrt{2}z}$  globally defined  on Sol space, the Riemannian submersion  given by
\begin{equation}\notag
\begin{array}{lll}
\pi:(\r^3,g_{Sol}= e^{2z}{\rm d}x^{2}+e^{-2z}{\rm d}y^{2}+{\rm
 d}z^{2})\to (\r^2, e^{-2z}{\rm d}y^{2}+{\rm
 d}z^{2}),\;
\pi(x,y,z)=(y,z)
\end{array}
\end{equation}
is a proper $f$-biharmonic map.

In fact, it is easy to  check that the orthonormal frame  $\{e_1=\frac{\partial}{\partial z},\;e_2=e^{z}\frac{\partial}{\partial y}, \;e_3=e^{-z}\frac{\partial}{\partial x}\}$ on Sol space is adapted  to the Riemannian submersion $\pi$ with
$e_3$ being vertical. A straightforward computation gives
\begin{equation}\label{PS1}
\begin{array}{lll}
[e_1, e_3] = -e_3,\;[e_2, e_3] = 0, \;[e_1, e_2] = e_2,\;f_1=\kappa_2=\sigma =0,\; -f_2= \kappa_1=-1,\\
({\rm and\; hence})\; K^N=e_1(f_2)-e_2(f_1)-f_1^2-f_2^2=-1.
\end{array}
\end{equation}
Substituting (\ref{PS1}) into (\ref{cor1}) yields
\begin{equation}\label{fTh1}
\begin{array}{lll}
f_{zz}+e^{-2z}f_{xx}-2f=0,\;
f_y=0.
\end{array}
\end{equation}
We find $f=Ae^{\sqrt{2}z}+Be^{-\sqrt{2}z}$ to be a special solution of (\ref{fTh1}),
where $A$ and $B$ are positive constants. The bitension field $\tau_2(\pi)=-2\frac{\partial}{\partial z}\neq0$ implies that $\pi$ is not biharmonic. Thus, we obtain the example.
\end{example}

\section {$f$-biharmonic  Riemannian submersions from $3$-space forms}
 A biharmonic Riemannian submersion from a 3-space form   into a surface  has to be harmonic (cf.\cite {WO}).
We want to know whether there is a proper $f$-biharmonic Riemannian submersion from a 3-space form.
Let   $\pi:(M^3(c),g)\to (N^2(c),h)$ be a Riemannian submersion from  a 3-space form of constant sectional curvature $c$ into a surface with constant Gauss curvature $c$. It follows  from  Lemma 3.2 in \cite{WO} that we can choose an orthonormal frame $\{e_1, e_2, e_3\} $ adapted to the Riemannian submersion with the integrability data $\{ f_1, f_2, \kappa_1, \kappa_2, \sigma\}$ and $\kappa_2 = 0$ such that
\begin{equation}\label{RSB1}
\begin{array}{ccc}
e_1(\sigma )-2\kappa_1\sigma=0,\;
e_1(\kappa_1)+\sigma^2-\kappa_{1}^2=c,\;
e_3(\sigma)=\kappa_{1}f_{1}=0,\;
K^N-3\sigma^2=c,\\e_3(\kappa_1)
=e_2(\sigma)=
e_2(\kappa_{1})=0,\;
\sigma^{2}-\kappa_1f_2=c,\;K^N=e_1(f_2)-e_2(f_1)-f_{1}^{2}-f_{2}^{2}.
\end{array}
\end{equation}
\begin{remark}\label{6}
Note that $\kappa_{1}=\kappa_{2}=0$, $\pi$ is harmonic and hence biharmonic. We now suppose that $\kappa_1\neq0$,\;$\kappa_2=0$ and the Gauss curvature of the base space $K^N=c$.  A simple computation using
(\ref{RSB1}) gives
\begin{equation}\label{RSB2}
\begin{array}{ccc}
\kappa_1\neq0,\;\sigma=f_{1}=e_2(\kappa_{1})=e_3(\kappa_1)=0,\\
\kappa_1f_2=-c,\;e_1(\kappa_1)=\kappa_{1}^2+c,\;
K^N=e_1(f_2)-f_{2}^{2}=c.
\end{array}
\end{equation}
Here, a Riemannian
submersion $\pi:(M^3(c),g)\to (N^2(c),h)$  may be the map  as $\pi_{U}:(M^3(c)\supseteq U,g)\to (V\subseteq N^2(c),h)$ from a subset $M^3(c)\supseteq U$ to  a subset $V\subseteq N^3(c)$. In spite of this, we denote  the
map $\pi_{U}:(M^3(c)\supseteq U,g)\to (V\subseteq N^2(c),h)$ as  $\pi:(M^3(c),g)\to (N^2(c),h)$ later in the rest of this section by abuse of notations.
\end{remark}
\begin{theorem}\label{Th2}
Let $\pi:(M^3(c),g)\to (N^2(c),h)$ be a Riemannian
submersion from a 3-space form of  constant sectional curvature $c$ into a surface of constant Gauss curvature $c$  with an adapted frame $\{e_1,\; e_2, \;e_3\}$ and the
integrability data $ f_1, f_2, \kappa_1,\;
\sigma$ and $\kappa_2=0$. Then, $\pi$ is proper $f$-biharmonic if and only if
 \begin{equation}\label{th2}
\begin{cases}
-\Delta^{M}(f\kappa_1) +f\kappa_1(-c+c^2/\kappa_1^2)=0,\\
ce_2(f)=0,\\
-\Delta^{M}(\kappa_1) +\kappa_1(-c+c^2/\kappa_1^2)\neq0.
\end{cases}
\end{equation}
\end{theorem}
\begin{proof}
Since the  Gauss curvature of the base space $K^{N}=c$, $\kappa_2=0$ and the assumption that $\pi$ is proper $f$-biharmonic, then we have $\kappa_1\neq0$. From these and
  Remark \ref{6}, we conclude that $ f_2=c/\kappa_1$ and hence $e_2(f_2)=0$.  Substituting this and (\ref{RSB2}) into (\ref{cor1}) we have $-\Delta^{M}(f\kappa_1) +f\kappa_1(-c+c^2/\kappa_1^2)=0$ and $ce_2(f)=0$. It is not difficult to check  that  the bitension field $\tau_2(\pi)=[-\Delta^{M}(\kappa_1) +\kappa_1(-c+c^2/\kappa_1^2)]\varepsilon_1$. Combining these, then $\pi$ is proper $f$-biharmonic  if and only if (\ref{th2}) holds. From which the theorem follows.
\end{proof}
Applying Theorem \ref{Th2} with $c=0$ we get
\begin{proposition}\label{r}
For the function $f=\rho>0$, then a Riemannian
submersion $\pi:(\r^3\backslash\{0\},{\rm d}\rho^2+{\rm d}z^2+\rho^2{\rm d}\theta^2)\to (\r^2\backslash\{0\},{\rm d}\rho^2+{\rm d}z^2)$  with $\pi(\rho,z,\theta)=(\rho,z)$ is a proper $f$-biharmonic map.
\end{proposition}
\begin{proof}
It is not difficult to check that the orthonormal frame  $\{e_1=\frac{\partial}{\partial\rho},\; e_2=\frac{\partial}{\partial z}, \;e_3=\frac{1}{\rho}\frac{\partial}{\partial\theta}\}$ on $(\r^3\backslash\{0\},{\rm d}\rho^2+{\rm d}z^2+\rho^2{\rm d}\theta^2)$
is adapted to the Riemannian submersion $\pi$ with ${\rm d}\pi(e_3 ) = 0$
and ${\rm d}\pi(e_i ) = \varepsilon_i, i = 1,2$, where $\{\varepsilon_1=\frac{\partial}{\partial\rho},\; \varepsilon_2=\frac{\partial}{\partial z}\}$  forms an orthonormal
frame on the base space $(\r^2\backslash\{0\},{\rm d}\rho^2+{\rm d}z^2)$. A straightforward computation gives
\begin{equation}\label{pr1}
\begin{array}{lll}
[e_1, e_3] = -\frac{1}{\rho} e_3, [e_2, e_3] = 0, [e_1, e_2] = 0,\;
f_1=f_2=\sigma=\kappa_2 = 0,\; \kappa_1=-\frac{1}{\rho}\neq0.
\end{array}
\end{equation}

Using (\ref{RC0}) and (\ref{pr1}) with $c=0$,  (\ref{th2}) reduces to
 \begin{equation}\label{pr2}
\begin{array}{ccc}
\Delta \left(\frac{f}{\rho}\right)=0.
\end{array}
\end{equation}
Looking for a special solution of the form  $f=f(\rho)$,  then (\ref{pr2}) becomes $
\left(\frac{f}{\rho}\right)''+\frac{1}{\rho}\left(\frac{f}{\rho}\right)'=0,$
which is solved by $
\frac{f}{\rho}=C_1+C_2\ln\rho,$
where $C_1, C_2$ are constants. Taking $C_1=1$ and $C_2=0$, we obtain $f=\rho$.
Thus, we obtain the proposition.
\end{proof}
Applying Theorem \ref{Th2} with $c=-1$ we have
\begin{proposition}\label{h}
For any positive constants\;$C_1$, $C_2$ and the function  $f=C_1e^{(1+\sqrt{3})\rho}+C_2e^{(1-\sqrt{3})\rho}$, then the Riemannian submersion $\pi:(H^3,{\rm d}\rho^2+e^{-2\rho}{\rm d}z^2+e^{-2\rho}{\rm d}\theta^2)\to (H^2,{\rm d}\rho^2+e^{-2\rho}{\rm d}z^2)$  with $\pi(\rho,z,\theta)=(\rho,z)$  is a proper $f$-biharmonic map.
\end{proposition}
\begin{proof}
We can easily check that the orthonormal frame  $\{e_1=\frac{\partial}{\partial\rho},\; e_2=e^{ \rho}\frac{\partial}{\partial z}, \;e_3=e^{ \rho}\frac{\partial}{\partial\theta}\}$ on $(H^3,{\rm d}\rho^2+e^{-2\rho}{\rm d}z^2+e^{-2\rho}{\rm d}\theta^2)$
is adapted to the Riemannian submersion $\pi$ with
${\rm d}\pi(e_3 ) = 0$. By a simple computation yields
\begin{equation}\label{p1}
[e_1, e_3] =  e_3, [e_2, e_3] = 0, [e_1, e_2] = e_2,\;f_1=\sigma=\kappa_2 = 0,\; \kappa_1=f_2=1.
\end{equation}
By  (\ref{RC0}) and (\ref{p1}), together with $c=-1$,  (\ref{th2}) reduces to
 \begin{equation}\label{p2}
\begin{array}{ccc}
-\Delta^{M}f +2f=0,\;
e_2(f)=0.
\end{array}
\end{equation}
 We can easily check that
$f=C_1e^{(1+\sqrt{3})\rho}+C_2e^{(1-\sqrt{3})\rho}$ satisfies (\ref{p2}),
where $C_1$  and $C_2$ are positive constants. Thus, we get the proposition.
\end{proof}
As an application of  Theorem \ref{Th2} with $c=1$, we have
\begin{proposition}\label{s}
For a positive function $f=f(\rho)$ defined on an open interval $I\subset(0,\frac{\pi}{2})$  solves the following $ODE$
\begin{equation}\label{S}
f''+\frac{2\cos2\rho-4}{\sin2\rho} f'+\frac{1+2\sin^2\rho}{\sin^2\rho}f=0,
\end{equation}
 then the Riemannian submersion $\pi:(S^3\supset I\times S^1\times S^1,{\rm d}\rho^2+\cos^2\rho{\rm d}z^2+\sin^2\rho{\rm d}\theta^2)\to (S^2\supset I\times S^1,{\rm d}\rho^2+\cos^2\rho{\rm d}z^2)$  with $\pi(\rho,z,\theta)=(\rho,z)$  is a proper $f$-biharmonic map.

\end{proposition}
\begin{proof}
One can easily check that the orthonormal frame  $\{e_1=\frac{\partial}{\partial\rho},\; e_2=\frac{1}{\cos \rho}\frac{\partial}{\partial z}, \;e_3=\frac{1}{\sin\rho}\frac{\partial}{\partial\theta}\}$ on $(I\times S^1\times S^1,{\rm d}\rho^2+\cos^2\rho{\rm d}z^2+\sin^2\rho{\rm d}\theta^2)$
is adapted to the Riemannian submersion  $\pi$ with
${\rm d}\pi(e_i ) = \varepsilon_i , i = 1,2$ and vertical $e_3$, where $\{\varepsilon_1=\frac{\partial}{\partial\rho},\; \varepsilon_2=\frac{1}{\cos\rho}\frac{\partial}{\partial z}\}$  forms an orthonormal
frame on the base space $(I\times S^1,{\rm d}\rho^2+\cos^2\rho{\rm d}z^2)$. A direct computation gives
\begin{equation}\label{pp0}
\begin{array}{lll}
[e_1, e_3] = -\cot\rho e_3, [e_2, e_3] = 0, [e_1, e_2] =\tan  \rho e_2,\\
f_1=\sigma=\kappa_2 = 0,\; \kappa_1=-\cot  \rho, f_2=\tan\rho.
\end{array}
\end{equation}
It is easy to check that $-\Delta^{M}(\kappa_1) +\kappa_1(-1+1/\kappa_1^2)=\frac{1+\sin^2\rho\cos2\rho}{\sin^3\rho\cos\rho}\neq0$. Substituting  (\ref{pp0}) and $c=1$ into  (\ref{th2}), then we have
 \begin{equation}\label{pp2}
\begin{array}{ccc}
-\Delta^{M}(-f\cot\rho) -f\cot\rho(-1+\tan^2\rho)=0\;{\rm and}\;
e_2(f)=0.
\end{array}
\end{equation}
If $f=f(\rho)$ depends on  only the variable $\rho$, then substituting this,  (\ref{RC0}) and (\ref{pp0}) into (\ref{pp2}) and simplifying the resulting equation we obtain (\ref{S}).
By the theory of $ODE$,  we conclude that there is  a local solution for  (\ref{S}). Moreover, if $y$  is a solution of  (\ref{S}) on  $I\subset(0,\frac{\pi}{2})$,  one finds $-y$ to be  a  solution of  (\ref{S}) on $I$. Hence, there is  a positive function  solution $f=f(\rho)$  on $I$ for  (\ref{S}). Thus,  we obtain the proposition.
\end{proof}

{\bf Acknowledgments:}
 Both authors would like to thank  Prof. Dr. Ye-Lin Ou for his guidance, help, and encouragement through many invaluable discussions, suggestions, and stimulating questions during the preparation of this work.  A part of the work was done when Ze-Ping Wang was a visiting scholar at Yunnan University in Fall 2023. He would like to express his gratitude to Professor Han-Chun Yang for his invitation and to Yunnan University for the hospitality.

\end{document}